\documentclass[11pt,a4paper]{article}

\usepackage[leqno]{amsmath}
\usepackage{graphicx,latexsym,euscript,makeidx,color}
\usepackage{amsmath,amsfonts,amssymb,amsthm,mathrsfs}
\usepackage{enumerate}

\usepackage[colorlinks,linkcolor=blue,anchorcolor=green,citecolor=red]{hyperref}

\usepackage{geometry}
\def\6n{\negthinspace \negthinspace \negthinspace \negthinspace \negthinspace \negthinspace }
\def\5n{\negthinspace \negthinspace \negthinspace \negthinspace \negthinspace }
\def\4n{\negthinspace \negthinspace \negthinspace \negthinspace }
\def\3n{\negthinspace \negthinspace \negthinspace }
\def\2n{\negthinspace \negthinspace }
\def\1n{\negthinspace }

   \def\cA{{\cal A}}

\def\dbE{\mathbb{E}}     
\def\dbF{\mathbb{F}}   \def\cF{{\cal F}}  
     
\def\dbH{\mathbb{H}}

   \def\cL{{\cal L}}  
   \def\cM{{\cal M}}  
   \def\cN{{\cal N}}  
   \def\cO{{\cal O}}  
\def\dbP{\mathbb{P}}     
   \def\cQ{{\cal Q}}  
\def\dbR{\mathbb{R}}     
\def\dbS{\mathbb{S}}     
     
   \def\cU{{\cal U}}

\def\ss{\smallskip}                \def\lt{\left}
\def\ms{\medskip}                \def\rt{\right}
               \def\lan{\langle}
\def\ds{\displaystyle}           \def\ran{\rangle}
\def\no{\noindent}        \def\q{\quad}                      \def\llan{\left\langle}
\def\ns{\noalign{\ss}}    \def\qq{\qquad}                    \def\rran{\right\rangle}
    \def\hb{\hbox}                     \def\blan{\big\langle}
                   \def\bran{\big\rangle}
         \def\rf{\eqref}                    \def\Blan{\Big\langle}
  \def\deq{\triangleq}               \def\Bran{\Big\rangle}
 \def\ae{\hbox{\rm a.e.}}           \def\({\Big (}
\def\les{\leqslant}       \def\as{\hbox{\rm a.s.}}           \def\){\Big )}
\def\ges{\geqslant}          \def\[{\Big[}
\def\cl{\overline}           \def\]{\Big]}
          \def\tr{\hbox{\rm tr$\,$}}         \def\cd{\cdot}
              
           \def\cl{\overline}

\def\a{\alpha}        \def\G{\Gamma}   \def\g{\gamma}
\def\b{\beta}         \def\D{\Delta}   \def\d{\delta}
           
\def\e{\varepsilon}     \def\l{\lambda}
    \def\Si{\Sigma}  
                
           \def\O{\Omega}  
          \def\f{\varphi}  \def\i{\infty}

\def\bde{\begin{definition}\label}    \def\ede{\end{definition}}
\def\bt{\begin{theorem}\label}        \def\et{\end{theorem}}
\def\bc{\begin{corollary}\label}      \def\ec{\end{corollary}}
\def\bl{\begin{lemma}\label}          \def\el{\end{lemma}}
\def\bp{\begin{proposition}\label}    \def\ep{\end{proposition}}
\def\bas{\begin{assumption}\label}    \def\eas{\end{assumption}}
\def\br{\begin{remark}\label}         \def\er{\end{remark}}
\def\bex{\begin{example}\label}       \def\ex{\end{example}}
\def\ba{\begin{array}}                \def\ea{\end{array}}
\def\be{\begin{equation}}
\def\bel{\begin{equation}\label}      \def\ee{\end{equation}}
\def\bea{\begin{eqnarray*}}           \def\eea{\end{eqnarray*}}

\newtheorem{theorem}{\indent Theorem}[section]
\newtheorem{definition}[theorem]{\indent Definition}
\newtheorem{proposition}[theorem]{\indent Proposition}
\newtheorem{corollary}[theorem]{\indent Corollary}
\newtheorem{lemma}[theorem]{\indent Lemma}
\newtheorem{remark}[theorem]{\indent Remark}
\newtheorem{example}[theorem]{\indent Example}

\newtheorem{assumption}[theorem]{\indent Assumption}

\makeatletter
   
   \@addtoreset{equation}{section}
\makeatother

\allowdisplaybreaks[4]

\begin{document}

\title{\bf Linear Quadratic Optimal Control Problems for Mean-Field Backward Stochastic
           Differential Equations}
\author{Xun Li\thanks{Department of Applied Mathematics, The Hong Kong Polytechnic University,
                      Hong Kong, China (malixun@polyu.edu.hk). This author was partially supported
                      by Hong Kong RGC under grants 15209614, 15224215 and 15255416.}~,\q
Jingrui Sun\thanks{Corresponding author, Department of Mathematics, National University of Singapore,
                   119076, Republic of Singapore (sjr@mail.ustc.edu.cn). This author was partially
                   supported by the National Natural Science Foundation of China (11401556) and the
                   Fundamental Research Funds for the Central Universities (WK 2040000012).}~,\q and \q
Jie Xiong\thanks{Department of Mathematics, University of Macau, Macau, China (jiexiong@umac.mo).
                 This author acknowledges the financial support from FDCT 025/2016/A1 and MYRG2014-00015-FST.}}


\maketitle

\no\bf Abstract: \rm
This paper is concerned with linear quadratic optimal control problems for mean-field backward
stochastic differential equations (MF-BSDEs, for short) with deterministic coefficients.
The optimality system, which is a linear mean-field forward-backward stochastic differential
equation with constraint, is obtained by a variational method.
By decoupling the optimality system, two coupled Riccati equations and an MF-BSDE are derived.
It turns out that the coupled two Riccati equations are uniquely solvable.
Then a complete and explicit representation is obtained for the optimal control.

\ms

\no\bf Key words: \rm
linear quadratic optimal control, mean-field backward stochastic differential equation,
Riccati equation, optimality system, decoupling

\ms

\no\bf AMS subject classifications. \rm 49N10, 49N35, 93E20

\section{Introduction}
The mean-field type stochastic control problem is importance in various fields such as
science, engineering, economics, management, and particularly in financial investment.
The theory of mean-field forward stochastic differential equations (MF-FSDEs, for short)
can be traced back to Kac \cite{Kac 1956} who presented the McKean-Vlasov stochastic
differential equation motivated by a stochastic toy model for the Vlasov kinetic equation
of plasma. Since then, research on related topics and their applications has become a
notable and serious endeavor among researchers in applied probability and optimal stochastic
controls, particularly in financial engineering. Typical representatives include,
but not limited to, McKean \cite{McKean 1966}, Dawson \cite{Dawson 1983}, Chan \cite{Chan 1994},
Buckdahn--Djehiche--Li--Peng \cite{Buckdahn-Djehiche-Li-Peng 2009},
Buckdahn--Li--Peng \cite{Buckdahn-Li-Peng 2009}, Borkar--Kumar \cite{Borkar-Kumar 2010},
Crisan--Xiong \cite{Crisan-Xiong 2010}, Andersson--Djehiche \cite{Andersson-Djehiche 2011},
Buckdahn--Djehiche--Li \cite{Buckdahn-Djehiche-Li 2011},
Meyer-Brandis--Oksendal--Zhou \cite{Meyer-Brandis-Oksendal-Zhou 2012},
Yong \cite{Yong 2013,Yong 2015}, Sun \cite{Sun 2016}, and Li--Sun--Yong \cite{Li-Sun-Yong 2016}.
The MF-FSDEs can be treated in a forward-looking way by starting with the initial state.
In financial investment, however, one frequently encounters financial investment problems
with future conditions (as random variables) specified. This naturally results in a
{\it mean-field backward stochastic differential equation} (MF-BSDE, for short) with a
given terminal condition (see Buckdahn--Djehiche--Li--Peng \cite{Buckdahn-Djehiche-Li-Peng 2009}
and Buckdahn--Li--Peng \cite{Buckdahn-Li-Peng 2009}).
This is an important and challenging research topic. Recently there has been increasing interest
in studying this type of stochastic control problems as well as their applications.
The optimal stochastic control problems under MF-BSDEs are underdeveloped in the literature, and
therefore many fundamental questions remain open and methodologies need to be significantly improved.

\ms

Let $(\O,\cF,\dbF,\dbP)$ be a complete filtered probability space on which a standard one-dimensional
Brownian motion $W=\{W(t); 0\les t<\i\}$ is defined, where $\dbF=\{\cF_t\}_{t\ges0}$ is the natural
filtration of $W$ augmented by all the $\dbP$-null sets in $\cF$. Consider the following controlled
linear MF-BSDE:
\bel{16Aug22-15:00}\left\{\2n\ba{ll}
\ds dY(s)=\Big\{A(s)Y(s)+\bar A(s)\dbE[Y(s)]+B(s)u(s)+\bar B(s)\dbE[u(s)]\\
\ns\hphantom{dY(s)=\Big\{}\ds+C(s)Z(s)+\bar C(s)\dbE[Z(s)]\Big\}ds+Z(s)dW(s),\q s\in[t,T],\\
\ns\ds Y(T)=\xi, \ea\right.\ee
where $A(\cd)$, $\bar A(\cd)$, $B(\cd)$, $\bar B(\cd)$, $C(\cd)$, $\bar C(\cd)$, $D(\cd)$, $\bar D(\cd)$
are given deterministic matrix-valued functions; $\xi$ is an $\cF_T$-measurable random vector; and $u(\cd)$
is the {\it control process}. The class of {\it admissible controls} for \rf{16Aug22-15:00} is
$$\cU[t,T]=\lt\{u:[t,T]\times\O\to\dbR^m\bigm|u(\cd)~\hb{is $\dbF$-progressively measurable,~}
\dbE\int_t^T|u(s)|^2ds<\i\rt\}.$$
Under some mild conditions on the coefficients of equation \rf{16Aug22-15:00}, for any terminal state
$\xi\in L^2_{\cF_T}(\O;\dbR^n)$ (the set of all $\cF_T$-measurable, square-integrable $\dbR^n$-valued
processes) and any admissible control $u(\cd)\in\cU[t,T]$, equation \rf{16Aug22-15:00} admits a unique
square-integrable adapted solution $(Y(\cd),Z(\cd))\equiv(Y(\cd\,;\xi,u(\cd)),Z(\cd\,;\xi,u(\cd)))$,
which is called the {\it state process} corresponding to $\xi$ and $u(\cd)$.
Now we introduce the following cost functional:
\bel{16Aug22-15:30}\ba{ll}
\ds J(t,\xi;u(\cd))\deq\dbE\bigg\{\lan GY(t),Y(t)\ran+\blan\bar G\dbE[Y(t)],\dbE[Y(t)]\bran\\
\ns\hphantom{J(t,\xi;u(\cd))\deq\dbE\bigg\{}\ds
+\int_t^T\[\lan Q(s)Y(s),Y(s)\ran+\blan\bar Q(s)\dbE[Y(s)],\dbE[Y(s)]\bran\\
\ns\hphantom{J(t,\xi;u(\cd))\deq\dbE\bigg\{+\int_t^T\[}\ds
+\lan R_1(s)Z(s),Z(s)\ran+\blan\bar R_1(s)\dbE[Z(s)],\dbE[Z(s)]\bran\\
\ns\hphantom{J(t,\xi;u(\cd))\deq\dbE\bigg\{+\int_t^T\[}\ds
+\lan R_2(s)u(s),u(s)\ran+\blan\bar R_2(s)\dbE[u(s)],\dbE[u(s)]\bran\]ds\bigg\},
\ea\ee
where $G$, $\bar G$ are symmetric matrices and $Q(\cd)$, $\bar Q(\cd)$, $R_i(\cd)$, $\bar R_i(\cd)$
$(i=1,2)$ are deterministic, symmetric matrix-valued functions. Our mean-field backward stochastic
linear quadratic (LQ, for short) optimal control problem can be stated as follows.

\ms

\bf Problem (MF-BSLQ). \rm For any given terminal state $\xi\in L^2_{\cF_T}(\O;\dbR^n)$,
find a $u^*(\cd)\in\cU[t,T]$ such that
\bel{16Aug22-15:50}J(t,\xi;u^*(\cd))=\inf_{u(\cd)\in\cU[t,T]}J(t,\xi;u(\cd))\deq V(t,\xi).\ee

Any $u^*(\cd)\in\cU[t,T]$ satisfying \rf{16Aug22-15:50} is called an {\it optimal control} of
Problem (MF-BSLQ) for the terminal state $\xi$, the corresponding $(Y^*(\cd),Z^*(\cd))
\equiv(Y(\cd\,;\xi,u^*(\cd)),Z(\cd\,;\xi,u^*(\cd)))$ is called an {\it optimal state process},
and the three-tuple $(Y^*(\cd),Z^*(\cd),u^*(\cd))$ is called an {\it optimal triple}.
The function $V(\cd\,,\cd)$ is called the {\it value function} of Problem (MF-BSLQ).
Note that when the mean-field part is absent, Problem (MF-BSLQ) is reduced to a stochastic LQ
optimal control of backward stochastic differential equations (see Lim--Zhou \cite{Lim-Zhou 2001}
and Zhang \cite{Zhang 2011} for some relevant results). For LQ optimal control problems of forward
stochastic differential equations, the interested reader is referred to, for examples,
\cite{Wonham 1968,Chen-Li-Zhou 1998,Ait Rami-Moore-Zhou 2001,Tang 2003,Sun-Li-Yong 2016} and the
book of Yong--Zhou \cite{Yong-Zhou 1999}.

\ms

The rest of the paper is organized as follows.
Section 2 gives some preliminaries. Among other things, we show Problem (MF-BSLQ) is uniquely
solvable from a Hilbert space viewpoint.
In section 3, we derive the optimality system by a variational method and the coupled two Riccati
equations by a decoupling technique.
Section 4 is devoted to the uniqueness and existence of solutions to the Riccati equations.
In Section 5, we present explicit formulas of the optimal controls and the value function.

\section{Preliminaries}

Throughout this paper, $\dbR^{n\times m}$ is the Euclidean space of all $n\times m$ real matrices,
$\dbS^n$ is the space of all symmetric $n\times n$ real matrices,
$\dbS^n_+$ is the subset of $\dbS^n$ consisting of positive definite matrices,
and $\cl{\dbS^n_+}$ is the closure of $\dbS^n_+$ in $\dbR^{n\times n}$.
When $m=1$, we simply write $\dbR^{n\times m}$ as $\dbR^n$, and when $n=m=1$, we drop the superscript.
Recall that the inner product $\lan\cd\,,\cd\ran$ on $\dbR^{n\times m}$ is given by
$\lan M,N\ran\mapsto\tr(M^\top N)$, where the superscript $\top$ denotes the transpose of matrices
and $\tr(K)$ denotes the trace of a matrix $K$,
and that the induced norm on $\dbR^{n\times m}$ is given by $|M|=\sqrt{\tr(M^\top M)}$.
If no confusion is likely, we shall use $\lan\cd\,,\cd\ran$ for inner products in possibly different
Hilbert spaces, and denote by $|\cd|$ the norm induced by $\lan\cd\,,\cd\ran$.
Let $t\in[0,T)$ and $\dbH$ be a given Euclidean space.
The space of $\dbH$-valued continuous functions on $[t,T]$ is denoted by $C([t,T];\dbH)$,
and the space of $\dbH$-valued, $p$th $(1\les p\les \i)$ power Lebesgue integrable functions on
$[t, T]$ is denoted by $L^p(t,T;\dbH)$.
Further, we introduce the following spaces of random variables and stochastic processes:
\begin{eqnarray*}
L^2_{\cF_T}(\O;\dbH)\3n&=&\3n
\Big\{\xi:\O\to\dbH\bigm|\xi~\hb{is $\cF_T$-measurable,}~\dbE|\xi|^2<\i\Big\},\\
L_\dbF^2(t,T;\dbH)\3n&=&\3n
\Big\{\f:[t,T]\times\O\to\dbH\bigm|\f(\cd)~\hb{is $\dbF$-progressively measurable,~}\\
&~&\hphantom{\Big\{\f:[t,T]\times\O\to\dbH\bigm|}
\dbE\int^T_t|\f(s)|^2ds<\i\Big\},\\
L_\dbF^2(\O;C([t,T];\dbH))\3n&=&\3n
\Big\{\f:[t,T]\times\O\to\dbH\bigm|\f(\cd)~\hb{is $\dbF$-adapted, continuous,~}\\
&~&\hphantom{\Big\{\f:[t,T]\times\O\to\dbH\bigm|}
\dbE\[\sup_{t\les s\les T}|\f(s)|^2\]<\i\Big\}.
\end{eqnarray*}

\ss

Next we introduce the following assumptions  that will be in force throughout this paper.

\ms

{\bf(H1)} The coefficients of the state equation satisfy the following:
$$A(\cd),\bar A(\cd)\in L^1(0,T;\dbR^{n\times n}),
~~B(\cd),\bar B(\cd)\in L^2(0,T;\dbR^{n\times m}),
~~C(\cd),\bar C(\cd)\in L^2(0,T;\dbR^{n\times n}).$$

{\bf(H2)} The weighting coefficients in the cost functional satisfy
$$\left\{\2n\ba{lll}
\ds G,\bar G\in\dbS^n, &\q Q(\cd),\bar Q(\cd)\in L^1(0,T;\dbS^n), \\
\ns R_1(\cd),\bar R_1(\cd)\in L^\i(0,T;\dbS^n), &\q R_2(\cd),\bar R_2(\cd)\in L^\i(0,T;\dbS^m),
\ea\right.$$
and there exists a constant $\d>0$ such that for $\ae~s\in[0,T]$,
$$\left\{\2n\ba{lll}
\ds G,\,G+\bar G\ges0,                &\q Q(s),\,Q(s)+\bar Q(s)\ges0,\\
\ns R_1(s),\,R_1(s)+\bar R_1(s)\ges0, &\q R_2(s),\,R_2(s)+\bar R_2(s)\ges\d I.
\ea\right.$$

Now we present a result concerning the well-posedness of the state equation \rf{16Aug22-15:00}.

\begin{theorem}\label{bt-16Aug19-17:00}\sl
Let {\rm(H1)} hold. Then for any $(\xi, u(\cd))\in L^2_{\cF_T}(\O;\dbR^n)\times \cU[t,T]$,
MF-BSDE \rf{16Aug22-15:00} admits a unique adapted solution
$(Y (\cd),Z(\cd))\in L_\dbF^2(\O;C([t,T];\dbR^n))\times L_\dbF^2(t,T;\dbR^n)$.
Moreover, there exists a constant $K>0$, independent of $\xi$ and $u(\cd)$, such that
$$\dbE\lt[\sup_{t\les s\les T}|Y(s)|^2+\int_t^T|Z(s)|^2ds\rt]
\les K\dbE\lt[|\xi|^2+\int_t^T|u(s)|^2ds\rt].$$
\end{theorem}

Note that (H1) allows the coefficients $A(\cd)$ and $C(\cd)$ to be unbounded, which is a little
different from the standard case \cite{Buckdahn-Djehiche-Li-Peng 2009,Buckdahn-Li-Peng 2009}.
However, the proof of Theorem \ref{bt-16Aug19-17:00} is similar to that of the case without
mean-field. We omit the proof here and refer the interested reader to
Sun--Yong \cite[Proposition 2.1]{Sun-Yong 2014} for details.

\ms

From Theorem \ref{bt-16Aug19-17:00}, one can easily see that under (H1)--(H2), Problem (MF-BSLQ)
makes sense. The following result tells us that under (H1)--(H2), Problem (MF-BSLQ) is actually
uniquely solvable for any terminal state $\xi\in L^2_{\cF_T}(\O;\dbR^n)$.

\begin{theorem}\label{bt-16Aug19-20:50}\sl
Let {\rm(H1)--(H2)} hold. Then for any terminal state $\xi\in L^2_{\cF_T}(\O;\dbR^n)$,
Problem {\rm(MF-BSLQ)} admits a unique optimal control.
\end{theorem}

\begin{proof}
For any $u(\cd)\in\cU[t,T]$, let $(Y^u(\cd),Z^u(\cd))$ be the unique adapted solution to
\bel{16Aug19-21:20}\left\{\2n\ba{ll}
\ds dY^u(s)=\Big\{A(s)Y^u(s)+\bar A(s)\dbE[Y^u(s)]+B(s)u(s)+\bar B(s)\dbE[u(s)]\\
\ns\hphantom{dY^u(s)=\Big\{}\ds
+C(s)Z^u(s)+\bar C(s)\dbE[Z^u(s)]\Big\}ds+Z^u(s)dW(s),\q s\in[t,T],\\
\ns\ds Y^u(T)=0. \ea\right.\ee
By the linearity of equation \rf{16Aug19-21:20} and Theorem \ref{bt-16Aug19-17:00},
we can define bounded linear operators
$\cL:\cU[t,T]\to L_\dbF^2(\O;C([t,T];\dbR^n))\times L_\dbF^2(t,T;\dbR^n)$
and $\cM:\cU[t,T]\to L^2_{\cF_t}(\O;\dbR^n)$ by $u(\cd)\mapsto (Y^u(\cd),Z^u(\cd))$
and $u(\cd)\mapsto Y^u(t)$, respectively, via the MF-BSDE \rf{16Aug19-21:20}.
Similarly, we can define bounded linear operators
$\cN:L^2_{\cF_T}(\O;\dbR^n)\to L_\dbF^2(\O;C([t,T];\dbR^n))\times L_\dbF^2(t,T;\dbR^n)$
and $\cO:L^2_{\cF_T}(\O;\dbR^n)\to L^2_{\cF_t}(\O;\dbR^n)$ by $\xi\mapsto (Y^\xi(\cd),Z^\xi(\cd))$
and  $\xi\mapsto Y^\xi(t)$, respectively, via the MF-BSDE
\bel{16Aug19-21:36}\left\{\2n\ba{ll}
\ds dY^\xi(s)=\Big\{A(s)Y^\xi(s)+\bar A(s)\dbE\big[Y^\xi(s)\big]
                   +C(s)Z^\xi(s)+\bar C(s)\dbE\big[Z^\xi(s)\big]\Big\}ds\\
\ns\ds\qq\qq\qq\qq\qq\qq~+Z^\xi(s)dW(s),\q s\in[t,T],\\
\ns\ds Y^\xi(T)=\xi. \ea\right.\ee
Then for any $(\xi, u(\cd))\in L^2_{\cF_T}(\O;\dbR^n)\times \cU[t,T]$, the corresponding
state process $(Y(\cd),Z(\cd))$ and the initial value $Y(t)$ are given by
$$(Y(\cd),Z(\cd))^\top=\cL u+\cN\xi,\qq Y(t)=\cM u+\cO\xi.$$
Now let $\cA^*$ denote the adjoint of an operator $\cA$, and define the bounded linear operator
$$\cQ\deq\begin{pmatrix}Q+\dbE^*\bar Q\dbE & 0\\0 & R_1+\dbE^*\bar R_1\dbE\end{pmatrix}$$
so that
$$\ba{ll}
\ds\dbE\int_t^T\[\lan Q(s)Y(s),Y(s)\ran+\blan\bar Q(s)\dbE[Y(s)],\dbE[Y(s)]\bran\\
\ns\hphantom{\dbE\int_t^T\[}\ds
+\lan R_1(s)Z(s),Z(s)\ran+\blan\bar R_1(s)\dbE[Z(s)],\dbE[Z(s)]\bran\]ds\\
\ns\ds=\llan\begin{pmatrix}Q+\dbE^*\bar Q\dbE & 0\\0 & R_1+\dbE^*\bar R_1\dbE\end{pmatrix}
            \begin{pmatrix}Y\\Z\end{pmatrix},\begin{pmatrix}Y\\Z\end{pmatrix}\rran\\
\ns\ds=\lan\cQ(\cL u+\cN\xi),\cL u+\cN\xi\ran,
\ea$$
where $\dbE:L^2_\dbF(t,T;\dbH)\to L^2(t,T;\dbH)$ is defined by $\dbE[Y](s)=\dbE[Y(s)]$.
Note that $\dbE^*:L^2(t,T;\dbH)\to L^2_\dbF(t,T;\dbH)$ is the adjoint operator. Thus,
\begin{eqnarray*}
J(t,\xi;u(\cd))\3n&=&\3n\blan(G+\dbE^*\bar G\dbE)(\cM u+\cO\xi),\cM u+\cO\xi\bran\\
\3n&~&\3n+\,\lan\cQ(\cL u+\cN\xi),\cL u+\cN\xi\ran+\blan(R_2+\dbE^*\bar R_2\dbE)u,u\bran\\
\3n&=&\3n\blan\big[\cM^*(G+\dbE^*\bar G\dbE)\cM+\cL^*\cQ\cL+(R_2+\dbE^*\bar R_2\dbE)\big]u,u\bran\\
\3n&~&\3n+\,2\blan\big[\cO^*(G+\dbE^*\bar G\dbE)\cM+\cN^*\cQ\cL\big]u,\xi\bran\\
\3n&~&\3n+\,\blan\big[\cO^*(G+\dbE^*\bar G\dbE)\cO+\cN^*\cQ\cN\big]\xi,\xi\bran.
\end{eqnarray*}
Note that under (H2), we have
$$G+\dbE^*\bar G\dbE=G-\dbE^*G\dbE+\dbE^*(G+\bar G)\dbE=(I-\dbE^*)G(I-\dbE)+\dbE^*(G+\bar G)\dbE\ges0.$$
Similarly, we can prove
$$\cQ\ges0,\qq R_2+\dbE^*\bar R_2\dbE\ges0.$$
Therefore, the map $u(\cd)\mapsto J(t,\xi;u(\cd))$ is convex and continuous.
Moreover, for any $u(\cd)\in\cU[t,T]$, we have by (H2):
\begin{eqnarray*}
&&\blan(R_2+\dbE^*\bar R_2\dbE)u,u\bran
=\dbE\int_t^T\Big\{\lan R_2(s)u(s),u(s)\ran+\blan\bar R_2(s)\dbE[u(s)],\dbE[u(s)]\bran\Big\}ds\\
&&={\d\over2}\dbE\int_t^T|u(s)|^2ds+\int_t^T\Big\{\dbE\lan [R_2(s)-\d/2]u(s),u(s)\ran
-\lan [R_2(s)-\d/2]\dbE[u(s)],\dbE[u(s)]\ran\Big\}ds\\
&&\hphantom{=}+\int_t^T\blan[R_2(s)+\bar R_2(s)-\d/2]\dbE[u(s)],\dbE[u(s)]\bran ds\\
&&\ges{\d\over2}\dbE\int_t^T|u(s)|^2ds.
\end{eqnarray*}
This further implies the map $u(\cd)\mapsto J(t,\xi;u(\cd))$ is strictly convex, and that
$$J(t,\xi;u(\cd))\to\i \q\hb{as}\q \dbE\int_t^T|u(s)|^2ds\to\i.$$
Therefore, by the basic theorem in convex analysis, for any given $\xi\in L^2_{\cF_T}(\O;\dbR^n)$,
Problem (MF-BSLQ) has a unique optimal control.
\end{proof}

\section{Optimality system, decoupling, and Riccati equations}

Let us first derive the optimality system for the optimal control of Problem (MF-BSLQ).
For simplicity of notation, in what follows we shall often suppress the time variable $s$
if no confusion can arise.

\bt{bt-16Aug22-16:06}\sl Let {\rm(H1)--(H2)} hold. Let $(Y^*(\cd),Z^*(\cd),u^*(\cd))$ be the
optimal triple for the terminal state $\xi\in L^2_{\cF_T}(\O;\dbR^n)$.
Then the solution $X^*(\cd)$ to the mean-field forward stochastic differential equation
(MF-FSDE, for short)
\bel{16Aug22-16:00}\left\{\2n\ba{ll}
\ds dX^*=\lt\{-A^\top X^*-\bar A^\top\dbE[X^*]+QY^*+\bar Q\dbE[Y^*]\rt\}ds\\
\ns\hphantom{dX^*=}\ds
+\lt\{-C^\top X^*-\bar C^\top\dbE[X^*]+R_1Z^*+\bar R_1\dbE[Z^*]\rt\}dW,\q s\in[t,T],\\
\ns\ds X^*(t)=GY^*(t)+\bar G\dbE[Y^*(t)],\ea\right.\ee
satisfies
\bel{16Aug22-16:10}R_2u^*+\bar R_2\dbE[u^*]-B^\top X^*-\bar B^\top\dbE[X^*]=0,\q\ae~s\in[t,T],~\as\ee
\end{theorem}

\begin{proof}
For any $u(\cd)\in\cU[t,T]$ and any $\e\in\dbR$, let $(Y(\cd),Z(\cd))$ be the solution of
$$\left\{\2n\ba{ll}
\ds dY=\lt\{AY+\bar A\dbE[Y]+Bu+\bar B\dbE[u]+CZ+\bar C\dbE[Z]\rt\}ds+ZdW,\q s\in[t,T],\\
\ns\ds Y(T)=0,\ea\right.$$
and let $(Y^\e(\cd),Z^\e(\cd))$ be the solution to the perturbed state equation
$$\left\{\2n\ba{ll}
\ds dY^\e=\lt\{AY^\e+\bar A\dbE[Y^\e]+B(u^*+\e u)+\bar B\dbE[u^*+\e u]
+CZ^\e+\bar C\dbE[Z^\e]\rt\}ds+Z^\e dW,\\
\ns\ds Y^\e(T)=\xi.\ea\right.$$
It is clear that $(Y^\e(\cd),Z^\e(\cd))=(Y^*(\cd)+\e Y(\cd),Z^*(\cd)+\e Z(\cd))$, and hence
\begin{eqnarray*}
&&J(t,\xi;u^*(\cd)+\e u(\cd))-J(t,\xi;u^*(\cd))\\
&&=2\e\dbE\bigg\{\lan GY^*(t),Y(t)\ran+\blan\bar G\dbE[Y^*(t)],\dbE[Y(t)]\bran
+\int_t^T\(\lan QY^*,Y\ran+\lan R_1Z^*,Z\ran+\lan R_2u^*,u\ran\)ds\\
&&\qq\qq+\int_t^T\(\blan\bar Q\dbE[Y^*],\dbE[Y]\bran+\lan\bar R_1\dbE[Z^*],\dbE[Z]\bran
+\blan\bar R_2\dbE[u^*],\dbE[u]\bran\)ds\bigg\}\\
&&\hphantom{=}+\e^2\dbE\bigg\{\lan GY(t),Y(t)\ran+\blan\bar G\dbE[Y(t)],\dbE[Y(t)]\bran
+\int_t^T\(\lan QY,Y\ran+\lan R_1Z,Z\ran+\lan R_2u,u\ran\)ds\\
&&\qq\qq+\int_t^T\(\blan\bar Q\dbE[Y],\dbE[Y]\bran+\lan\bar R_1\dbE[Z],\dbE[Z]\bran
+\blan\bar R_2\dbE[u],\dbE[u]\bran\)ds\bigg\}.
\end{eqnarray*}
Applying It\^o's formula to $s\mapsto\lan X^*(s),Y(s)\ran$, we have
\begin{eqnarray*}
&&-\dbE\big\{\lan GY^*(t),Y(t)\ran+\blan\bar G\dbE[Y^*(t)],\dbE[Y(t)]\bran\big\}
=-\dbE\blan GY^*(t)+\bar G\dbE[Y^*(t)],Y(t)\bran\\
&&=\dbE\int_t^T\Big\{\blan QY^*+\bar Q\dbE[Y^*],Y\bran+\blan R_1Z^*+\bar R_1\dbE[Z^*],Z\bran
+\blan B^\top X^*+\bar B^\top\dbE[X^*],u\bran\Big\}ds.
\end{eqnarray*}
It follows that for any $u(\cd)\in\cU[t,T]$ and any $\e\in\dbR$,
\begin{eqnarray*}
&&J(t,\xi;u^*(\cd)+\e u(\cd))-J(t,\xi;u^*(\cd))\\
&&=2\e\dbE\int_t^T\blan R_2u^*+\bar R_2\dbE[u^*]-B^\top X^*-\bar B^\top\dbE[X^*],u\bran ds\\
&&\hphantom{=}+\e^2\dbE\bigg\{\lan GY(t),Y(t)\ran+\blan\bar G\dbE[Y(t)],\dbE[Y(t)]\bran
+\int_t^T\(\lan QY,Y\ran+\lan R_1Z,Z\ran+\lan R_2u,u\ran\)ds\\
&&\qq\qq~+\int_t^T\(\blan\bar Q\dbE[Y],\dbE[Y]\bran+\lan\bar R_1\dbE[Z],\dbE[Z]\bran
+\blan\bar R_2\dbE[u],\dbE[u]\bran\)ds\bigg\}.
\end{eqnarray*}
Since $u^*(\cd)$ is the optimal control of Problem (MF-BSLQ) for the terminal state $\xi$,
dividing by $\e$ in the above and then letting $\e\to0$, we obtain
$$\dbE\int_t^T\blan R_2u^*+\bar R_2\dbE[u^*]-B^\top X^*-\bar B^\top\dbE[X^*],u\bran ds=0,
\q\forall u(\cd)\in\cU[t,T],$$
from which \rf{16Aug22-16:10} follows immediately.
\end{proof}

From the above result, we see that if $u(\cd)$ happens to be an optimal control of
Problem (MF-BSLQ) for terminal state $\xi$, then the following mean-field forward-backward
stochastic differential equation (MF-FBSDE, for short) admits an adapted solution
$(X(\cd),Y(\cd),Z(\cd))$:
\bel{16Aug22-16:20}\left\{\2n\ba{ll}
\ds dX=\lt\{-A^\top X-\bar A^\top\dbE[X]+QY+\bar Q\dbE[Y]\rt\}ds\\
\ns\hphantom{dX=}\ds+\lt\{-C^\top X-\bar C^\top\dbE[X]+R_1Z+\bar R_1\dbE[Z]\rt\}dW,\q s\in[t,T],\\
\ns\ds dY=\Big\{AY+\bar A\dbE[Y]+Bu+\bar B\dbE[u]+CZ+\bar C\dbE[Z]\Big\}ds+ZdW,\q s\in[t,T],\\
\ns\ds X(t)=GY(t)+\bar G\dbE[Y(t)], \qq Y(T)=\xi,
\ea\right.\ee
and the following {\it stationarity condition} holds:
\bel{16Aug22-16:30}R_2u+\bar R_2\dbE[u]-B^\top X-\bar B^\top\dbE[X]=0,\q\ae~s\in[t,T],~\as\ee
We call \rf{16Aug22-16:20}, together with the stationarity condition \rf{16Aug22-16:30},
the {\it optimality system} for the optimal control of Problem (MF-BSLQ).
Note that \rf{16Aug22-16:30} brings a coupling into the MF-FBSDE \rf{16Aug22-16:20}
and does not provide a representation for $u(\cd)$ because the equation for $X(\cd)$ involves
$Y(\cd)$ and $Z(\cd)$.

\ms

To solve the optimality system \rf{16Aug22-16:20}--\rf{16Aug22-16:30}, we use the decoupling
technique inspired by the four-step scheme introduced in \cite{Ma-Protter-Yong 1994,Ma-Yong 1994}
for general FBSDEs. This will lead to a derivation of two Riccati equations.
To be precise, we conjecture that $X(\cd)$ and $Y(\cd)$ are related by the following:
\bel{16Aug22-17:00}Y(s)=-\Si(s)\big\{X(s)-\dbE[X(s)]\big\}-\G(s)\dbE[X(s)]-\f(s),\q s\in[t,T],\ee
where $\Si(\cd),\G(\cd):[0,T]\to\dbS^n$ are absolutely continuous and $\f(\cd)$ satisfies
\bel{16Aug22-17:10}d\f(s)=\a(s) ds+\b(s) dW(s),\qq \f(T)=-\xi,\ee
for some $\dbF$-progressively measurable processes $\a(\cd)$ and $\b(\cd)$. Note that
\bel{16Aug22-17:22}\left\{\2n\ba{ll}
\ds d\dbE[X]=\lt\{-\big(A+\bar A\big)^\top\dbE[X]+\big(Q+\bar Q\big)\dbE[Y]\rt\}ds\\
\ns d\dbE[Y]=\Big\{\big(A+\bar A\big)\dbE[Y]+\big(B+\bar B\big)\dbE[u]+\big(C+\bar C\big)\dbE[Z]\Big\}ds,\\
\ns \dbE[X(t)]=\big(G+\bar G\big)\dbE[Y(t)],\qq \dbE[Y(T)]=\dbE[\xi],\\
\ns \big(R_2+\bar R_2\big)\dbE[u]-\big(B+\bar B\big)^\top \dbE[X]=0.
\ea\right.\ee
Thus,
\bel{16Aug22-17:38-1}\left\{\2n\ba{ll}
\ds d\big(X-\dbE[X]\big)=\lt\{-A^\top\big(X-\dbE[X]\big)+Q\big(Y-\dbE[Y]\big)\rt\}ds\\
\ns\hphantom{d\big(X-\dbE[X]\big)=}+\Big\{\1n-C^\top X-\bar C^\top\dbE[X]+R_1Z+\bar R_1\dbE[Z]\Big\}dW,\\
\ns d\big(Y-\dbE[Y]\big)=\Big\{A\big(Y-\dbE[Y]\big)+B\big(u-\dbE[u]\big)+C\big(Z-\dbE[Z]\big)\Big\}ds+ZdW,\\
\ns X(t)-\dbE[X(t)]=G\big(Y(t)-\dbE[Y(t)]\big),\q~ Y(T)-\dbE[Y(T)]=\xi-\dbE[\xi],\\
\ns R_2\big(u-\dbE[u]\big)-B^\top\big(X-\dbE[X]\big)=0.
\ea\right.\ee
From \rf{16Aug22-17:00} we have
\bel{16Aug22-17:38-2}Y-\dbE[Y]=-\Si\big(X-\dbE[X]\big)-\big(\f-\dbE[\f]\big),\qq\dbE[Y]=-\G\dbE[X]-\dbE[\f].\ee
Denoting $\eta(\cd)=\f(\cd)-\dbE[\f(\cd)]$ and $\g(\cd)=\a(\cd)-\dbE[\a(\cd)]$,
we have from \rf{16Aug22-17:10} that
\bel{16Aug22-17:38-3}d\eta(s)=\g(s) ds+\b(s) dW(s),\qq \eta(T)=\dbE[\xi]-\xi.\ee
Then \rf{16Aug22-17:38-1}--\rf{16Aug22-17:38-3} yield
\begin{eqnarray*}
0 \3n&=&\3n d(Y-\dbE[Y])+\dot\Si(X-\dbE[X])ds+\Si d(X-\dbE[X])+d\eta \\
\3n&=&\3n \big\{A(Y-\dbE[Y])+B(u-\dbE[u])+C(Z-\dbE[Z])\big\}ds+ZdW  \\
\3n&~&\3n +\,\dot\Si(X-\dbE[X])ds+\big\{\1n-\Si A^\top(X-\dbE[X])+\Si Q(Y-\dbE[Y])\big\}ds\\
\3n&~&\3n +\,\big\{\1n-\Si C^\top X-\Si\bar C^\top\dbE[X]+\Si R_1Z+\Si\bar R_1\dbE[Z]\big\}dW+\g ds+\b dW\\
\3n&=&\3n \big\{A(Y-\dbE[Y])+B(u-\dbE[u])+C(Z-\dbE[Z])+\dot\Si(X-\dbE[X]) \\
\3n&~&\3n \qq~-\Si A^\top(X-\dbE[X])+\Si Q(Y-\dbE[Y])+\g\big\}ds\\
\3n&~&\3n +\,\big\{Z-\Si C^\top X-\Si\bar C^\top\dbE[X]+\Si R_1Z+\Si\bar R_1\dbE[Z]+\b\big\}dW \\
\3n&=&\3n \big\{-A\Si(X-\dbE[X])-A\eta+BR_2^{-1}B^\top(X-\dbE[X])+C(Z-\dbE[Z]) \\
\3n&~&\3n \qq~+\dot\Si(X-\dbE[X])-\Si A^\top(X-\dbE[X])-\Si Q\Si(X-\dbE[X])-\Si Q\eta+\g\big\}ds\\
\3n&~&\3n +\,\big\{Z-\Si C^\top X-\Si\bar C^\top\dbE[X]+\Si R_1Z+\Si\bar R_1\dbE[Z]+\b\big\}dW \\
\3n&=&\3n \big\{\big(\dot\Si-A\Si-\Si A^\top-\Si Q\Si+BR_2^{-1}B^\top\big)(X-\dbE[X])\\
\3n&~&\3n \qq~+C(Z-\dbE[Z])-(A+\Si Q)\eta+\g\big\}ds\\
\3n&~&\3n +\,\big\{Z-\Si C^\top X-\Si\bar C^\top\dbE[X]+\Si R_1Z+\Si\bar R_1\dbE[Z]+\b\big\}dW.
\end{eqnarray*}
This implies
\begin{eqnarray}
\label{16Aug22-20:00-1}
&\big(\dot\Si\1n-\1nA\Si\1n-\1n\Si A^\top\1n-\1n\Si Q\Si\1n+\1nBR_2^{-1}B^\top\big)
(X\1n-\1n\dbE[X])\1n+\1nC(Z\1n-\1n\dbE[Z])\1n-\1n(A\1n+\1n\Si Q)\eta+\g=0,&\\
\label{16Aug22-20:00-2}
&Z-\Si C^\top X-\Si\bar C^\top\dbE[X]+\Si R_1Z+\Si\bar R_1\dbE[Z]+\b=0.&
\end{eqnarray}
Now from \rf{16Aug22-20:00-2} we have
\bel{16Jan12-diffu-1}(I+\Si R_1+\Si\bar R_1)\dbE[Z]-\Si(C+\bar C)^\top\dbE[X]+\dbE[\b]=0.\ee
Subtracting \rf{16Jan12-diffu-1} from \rf{16Aug22-20:00-2}, we obtain
\bel{16Jan12-diffu-2}(I+\Si R_1)(Z-\dbE[Z])-\Si C^\top(X-\dbE[X])+(\b-\dbE[\b])=0.\ee
Assuming that $I+\Si R_1$ and $I+\Si R_1+\Si\bar R_1$ are invertible,
we obtain from \rf{16Jan12-diffu-1} and \rf{16Jan12-diffu-2}:
\begin{eqnarray}
\label{16Jan12-diffu-3}
\dbE[Z]&\3n=\3n&(I+\Si R_1+\Si\bar R_1)^{-1}\big\{\Si(C+\bar C)^\top\dbE[X]-\dbE[\b]\big\},\\
\label{16Jan12-diffu-4}
Z-\dbE[Z]&\3n=\3n&(I+\Si R_1)^{-1}\big\{\Si C^\top(X-\dbE[X])-(\b-\dbE[\b])\big\}.
\end{eqnarray}
Substitution of \rf{16Jan12-diffu-4} into \rf{16Aug22-20:00-1} now gives
\begin{eqnarray*}
&&\lt[\dot\Si-A\Si-\Si A^\top-\Si Q\Si+BR_2^{-1}B^\top+C(I+\Si R_1)^{-1}\Si C^\top\rt](X-\dbE[X])\\
&&\,-\,C(I+\Si R_1)^{-1}(\b-\dbE[\b])-(A+\Si Q)\eta+\g=0,
\end{eqnarray*}
from which one should let
\bel{16Aug22-21:08}\left\{\2n\ba{ll}
\ds \dot\Si-A\Si-\Si A^\top-\Si Q\Si+BR_2^{-1}B^\top+C(I+\Si R_1)^{-1}\Si C^\top=0,\\
\ns \g-C(I+\Si R_1)^{-1}(\b-\dbE[\b])-(A+\Si Q)\eta=0.
\ea\right.\ee
Also, we have from \rf{16Aug22-17:22}, \rf{16Aug22-17:38-2}, and \rf{16Jan12-diffu-3}:
\begin{eqnarray*}
0 &\3n=\3n& {d\over ds}\big(\dbE[Y]+\G\dbE[X]+\dbE[\f]\big)\\
&\3n=\3n& (A+\bar A)\dbE[Y]+(B+\bar B)\dbE[u]+(C+\bar C)\dbE[Z]\\
&\3n~\3n& +\,\dot\G\dbE[X]-\G(A+\bar A)^\top\dbE[X]+\G(Q+\bar Q)\dbE[Y]+\dbE[\a]\\
&\3n=\3n& -\,(A+\bar A)\G\dbE[X]-(A+\bar A)\dbE[\f]+(B+\bar B)(R_2+\bar R_2)^{-1}(B+\bar B)^\top\dbE[X]\\
&\3n~\3n& +\,(C+\bar C)(I+\Si R_1+\Si\bar R_1)^{-1}\big\{\Si(C+\bar C)^\top\dbE[X]-\dbE[\b]\big\}\\
&\3n~\3n& +\,\dot\G\dbE[X]-\G(A+\bar A)^\top\dbE[X]-\G(Q+\bar Q)\G\dbE[X]-\G(Q+\bar Q)\dbE[\f]+\dbE[\a]\\
&\3n=\3n& \big\{\dot\G-(A+\bar A)\G-\G(A+\bar A)^\top-\G(Q+\bar Q)\G
+(B+\bar B)(R_2+\bar R_2)^{-1}(B+\bar B)^\top\\
&\3n~\3n&\hphantom{\big\{\dot\G}+(C+\bar C)(I+\Si R_1+\Si\bar R_1)^{-1}\Si(C+\bar C)^\top\big\}\dbE[X]\\
&\3n~\3n& -\,\big[(A+\bar A)+\G(Q+\bar Q)\big]\dbE[\f]-(C+\bar C)(I+\Si R_1+\Si\bar R_1)^{-1}\dbE[\b]+\dbE[\a].
\end{eqnarray*}
Hence, one should let
\bel{16Aug22-21:15}\left\{\2n\ba{ll}
\ds\dot\G-(A+\bar A)\G-\G(A+\bar A)^\top-\G(Q+\bar Q)\G+(B+\bar B)(R_2+\bar R_2)^{-1}(B+\bar B)^\top\\
\ns\hphantom{\dot\G}+(C+\bar C)(I+\Si R_1+\Si\bar R_1)^{-1}\Si(C+\bar C)^\top=0,\\
\ns\dbE[\a]-\big[(A+\bar A)+\G(Q+\bar Q)\big]\dbE[\f]-(C+\bar C)(I+\Si R_1+\Si\bar R_1)^{-1}\dbE[\b]=0.
\ea\right.\ee
Moreover, comparing the terminal values on both sides of the two equations in \rf{16Aug22-17:38-2},
one has
$$\Si(T)=0,\qq \G(T)=0.$$
Therefore, by \rf{16Aug22-21:08}--\rf{16Aug22-21:15}, we see that $\Si(\cd)$ and $\G(\cd)$
should satisfy the following Riccati-type equations, respectively:
\begin{eqnarray}
&&\label{16Aug22-Ric-Si}\left\{\2n\ba{ll}
\ds\dot\Si-A\Si-\Si A^\top-\Si Q\Si+BR_2^{-1}B^\top+C(I+\Si R_1)^{-1}\Si C^\top=0,\q s\in[0,T],\\
\ns\Si(T)=0, \ea\right.\\
\ns&&\label{16Aug22-Ric-G}\left\{\2n\ba{ll}
\ds\dot\G-(A+\bar A)\G-\G(A+\bar A)^\top-\G(Q+\bar Q)\G+(B+\bar B)(R_2+\bar R_2)^{-1}(B+\bar B)^\top\\
\ns\hphantom{\dot\G}+(C+\bar C)(I+\Si R_1+\Si\bar R_1)^{-1}\Si(C+\bar C)^\top=0,\q s\in[0,T],\\
\ns\G(T)=0, \ea\right.
\end{eqnarray}
and $\f(\cd)$ should satisfy the following MF-BSDE on $[0,T]$:
\bel{16Aug22-f}\left\{\2n\ba{ll}
\ds d\f=\Big\{(A+\Si Q)\f+\big[\bar A+\G(Q+\bar Q)-\Si Q\big]\dbE[\f]+C(I+\Si R_1)^{-1}\b\\
\ns\hphantom{d\f=\Big\{}
+\big[(C+\bar C)(I+\Si R_1+\Si\bar R_1)^{-1}-C(I+\Si R_1)^{-1}\big]\dbE[\b]\Big\}ds+\b dW,\\
\ns \f(T)=-\xi. \ea\right.\ee

\section{Unique solvability of Riccati equatuions}

In this section we shall establish the unique solvability of the Riccati equations
\rf{16Aug22-Ric-Si} and \rf{16Aug22-Ric-G}. Once $\Si(\cd)$ and $\Pi(\cd)$ are known,
the existence of a solution to MF-BSDE \rf{16Aug22-f} will immediately follows from
Theorem \ref{bt-16Aug19-17:00}.

\begin{theorem}\label{bt-16Mar18-21:30}\sl
Let {\rm(H1)--(H2)} hold. Then the Riccati equations \rf{16Aug22-Ric-Si} and \rf{16Aug22-Ric-G}
admit unique solutions $\Si(\cd)\in C([0,T];\cl{\dbS^n_+})$ and $\G(\cd)\in C([0,T];\cl{\dbS^n_+})$,
respectively.
\end{theorem}

\begin{proof}
For $\l>0$ and $\e\ges0$, let us consider the forward stochastic differential equation
(FSDE, for short)
$$\left\{\2n\ba{ll}
\ds dX(s)=\Big\{A(s)X(s)+\bar A(s)\dbE[X(s)]+B(s)u(s)+\bar B(s)\dbE[u(s)]\\
\ns\hphantom{dX(s)=\Big\{}
+C(s)v(s)+\bar C(s)\dbE[v(s)]\Big\}ds+v(s)dW(s),\q s\in[t,T],\\
\ns X(t)=\xi,\ea\right.$$
and the cost functional
\begin{eqnarray*}
&&J_{\l,\e}(t,\xi;u(\cd),v(\cd))\\
&&\deq\dbE\bigg\{\int_t^T\[\lan Q(s)X(s),X(s)\ran+\blan\bar Q(s)\dbE[X(s)],\dbE[X(s)]\bran\\
&&\hphantom{\deq\dbE\bigg\{\int_t^T\[}
+\lan [\e I+R_1(s)]v(s),v(s)\ran+\blan\bar R_1(s)\dbE[v(s)],\dbE[v(s)]\bran\\
&&\hphantom{\deq\dbE\bigg\{\int_t^T\[}
+\lan R_2(s)u(s),u(s)\ran+\blan\bar R_2(s)\dbE[u(s)],\dbE[u(s)]\bran\]ds+\l |X(T)|^2\bigg\}.
\end{eqnarray*}
We pose the following forward mean-field LQ problem: For any given initial pair
$(t,\xi)\in[0,T]\times L^2_{\cF_t}(\O;\dbR^n)$, find a pair
$(u^*(\cd),v^*(\cd))\in L_\dbF^2(t,T;\dbR^m)\times L_\dbF^2(t,T;\dbR^n)$ such that
$$J_{\l,\e}(t,\xi;u^*(\cd),v^*(\cd))
=\inf_{u(\cd),v(\cd)}J_{\l,\e}(t,\xi;u(\cd),v(\cd))\deq V_{\l,\e}(t,\xi)$$
as $(u(\cd),v(\cd))$ ranges over the space $L_\dbF^2(t,T;\dbR^m)\times L_\dbF^2(t,T;\dbR^n)$.
By (H2), we have for any $(t,\xi)\in[0,T]\times L^2_{\cF_t}(\O;\dbR^n)$
and any $(u(\cd),v(\cd))\in L_\dbF^2(t,T;\dbR^m)\times L_\dbF^2(t,T;\dbR^n)$,
\bel{16Aug24-15:45}\ba{ll}
\ds J_{\l,\e}(t,\xi;u(\cd),v(\cd))\\
\ns\ds\ges\int_t^T\Big\{\lan Q(s)\dbE[X(s)],\dbE[X(s)]\ran+\blan\bar Q(s)\dbE[X(s)],\dbE[X(s)]\bran\\
\ns\ds\hphantom{\ges\int_t^T\Big\{}
+\lan R_1(s)\dbE[v(s)],\dbE[v(s)]\ran+\blan\bar R_1(s)\dbE[v(s)],\dbE[v(s)]\bran\\
\ns\ds\hphantom{\ges\int_t^T\Big\{}
+\lan [R_2(s)-\d/2]\dbE[u(s)],\dbE[u(s)]\ran+\blan\bar R_2(s)\dbE[u(s)],\dbE[u(s)]\bran\Big\}ds\\
\ns\ds\hphantom{\ges}
+\e\dbE\int_t^T|v(s)|^2ds+{\d\over2}\dbE\int_t^T|u(s)|^2ds\\
\ns\ds\ges\left(\e\wedge{\d\over2}\right)\dbE\int_t^T\[|v(s)|^2+|u(s)|^2\]ds.
\ea\ee
Then it follows from \cite[Theorem 5.2]{Sun 2016} (see also \cite[Theorem 4.1]{Yong 2013})
that for any $\l,\e>0$, the following two Riccati equations
\bel{16Mar21-P_e}\left\{\2n\ba{ll}
\ds\dot P_{\l,\e}+P_{\l,\e} A+A^\top P_{\l,\e}+Q
-P_{\l,\e}(B,C)\begin{pmatrix}R_2&0\\0&\e I+R_1+P_{\l,\e}\end{pmatrix}^{-1}(B,C)^\top P_{\l,\e}=0,\\
\ns\ds P_{\l,\e}(T)=\l I,\ea\right.\ee
and
\bel{16Mar21-Pi_e}\left\{\2n\ba{ll}
\ds \dot\Pi_{\l,\e}+\Pi_{\l,\e}(A+\bar A)+(A+\bar A)^\top\Pi_{\l,\e}+Q+\bar Q\\
\ns\ds-\,\Pi_{\l,\e}(B+\bar B,C+\bar C)
\begin{pmatrix}R_2\1n+\1n\bar R_2&0\\0&\e I\1n+\1nR_1\1n+\1n\bar R_1\1n+\1nP_{\l,\e}\end{pmatrix}^{-1}
(B+\bar B,C+\bar C)^\top\Pi_{\l,\e}=0,\\
\ns\ds \Pi_{\l,\e}(T)=\l I, \ea\right.\ee
admit unique solutions $P_{\l,\e}(\cd)$ and $\Pi_{\l,\e}(\cd)$, respectively, such that
\bel{16Mar20-21:40}\ba{rr}
\ds V_{\l,\e}(t,\xi)=\dbE\lan P_{\l,\e}(t)(\xi-\dbE[\xi]),\xi-\dbE[\xi]\ran
+\lan\Pi_{\l,\e}(t)\dbE[\xi],\dbE[\xi]\ran,\\
\ns\ds\forall(t,\xi)\in[0,T]\times L^2_{\cF_t}(\O;\dbR^n).\ea\ee
For fixed $\l>0$, we have
\bel{16Aug24-16:10}\ba{rr}
\ds V_{\l,\e}(t,\xi)=\inf_{u(\cd),v(\cd)}J_{\l,\e}(t,\xi;u(\cd),v(\cd))
\les\inf_{u(\cd),v(\cd)}J_{\l,\e^\prime}(t,\xi;u(\cd),v(\cd))=V_{\l,\e^\prime}(t,\xi),\\
\ns\ds\forall (t,\xi)\in[0,T]\times L^2_{\cF_t}(\O;\dbR^n),
\ea\ee
whenever $0\les\e\les\e^\prime$. This implies
\bel{16Aug24-16:15}P_{\l,\e}(t)\les P_{\l,\e^\prime}(t),\q \Pi_{\l,\e}(t)\les \Pi_{\l,\e^\prime}(t),
\qq \forall\, t\in[0,T];\q \forall\,0<\e\les\e^\prime.\ee
On the other hand, we may conclude from \rf{16Aug24-15:45} that
$$V_{\l,0}(t,\xi)>0,\qq\forall (t,\xi)\in[0,T]\times L^2_{\cF_t}(\O;\dbR^n)\hb{~with~}\xi\ne0,$$
which, together with \rf{16Aug24-16:10} and \rf{16Aug24-16:15}, implies that
the limits $\lim_{\e\to0}P_{\l,\e}(t)$ and $\lim_{\e\to0}\Pi_{\l,\e}(t)$ exist, and
$$P_\l(t)\deq\lim_{\e\to0}P_{\l,\e}(t)>0,\q\Pi_\l(t)\deq\lim_{\e\to0}\Pi_{\l,\e}(t)>0,
\qq\forall t\in[0,T].$$
By \rf{16Mar21-P_e}, we get
\begin{eqnarray*}
&& P_{\l,\e}(t)=\l I+\int_t^T\Bigg[P_{\l,\e}A+A^\top P_{\l,\e}+Q \\
&&\hphantom{P_{\l,\e}(t)=\l I+\int_t^T\Bigg[}
-P_{\l,\e}(B,C)\begin{pmatrix}R_2&0\\0&\e I+R_1+P_{\l,\e}\end{pmatrix}^{-1}(B,C)^\top P_{\l,\e}\Bigg]ds.
\end{eqnarray*}
Passing to limit as $\e\to0$, by the bounded convergence theorem, we have
$$P_\l(t)=\l I+\int_t^T\Bigg[P_\l A+A^\top P_\l+Q
-P_\l(B,C)\begin{pmatrix}R_2&0\\0&R_1+P_\l\end{pmatrix}^{-1}(B,C)^\top P_\l\Bigg]ds.$$
Therefore,
$$\left\{\2n\ba{ll}
\ds\dot P_\l+P_\l A+A^\top P_\l+Q
-P_\l(B,C)\begin{pmatrix}R_2&0\\0&R_1+P_\l\end{pmatrix}^{-1}(B,C)^\top P_\l=0,\\
\ns\ds P_\l(T)=\l I.\ea\right.$$
Similarly using \rf{16Mar21-Pi_e}, we have
$$\left\{\2n\ba{ll}
\ds \dot\Pi_\l+\Pi_\l(A+\bar A)+(A+\bar A)^\top\Pi_\l+Q+\bar Q\\
\ns\ds\hphantom{\dot\Pi_\l}
-\Pi_\l(B+\bar B,C+\bar C)
\begin{pmatrix}R_2+\bar R_2&0\\0&R_1+\bar R_1+P_\l\end{pmatrix}^{-1}
(B+\bar B,C+\bar C)^\top\Pi_\l=0,\\
\ns\ds \Pi_\l(T)=\l I. \ea\right.$$
Next, for fixed $\e>0$, we have
\begin{eqnarray*}
V_{\l,\e}(t,\xi)=\inf_{u(\cd),v(\cd)}J_{\l,\e}(t,\xi;u(\cd),v(\cd))
\les\inf_{u(\cd),v(\cd)}J_{\l^\prime,\e}(t,\xi;u(\cd),v(\cd))=V_{\l^\prime,\e}(t,\xi),&&\\
\forall (t,\xi)\in[0,T]\times L^2_{\cF_t}(\O;\dbR^n),&&
\end{eqnarray*}
whenever $0<\l\les\l^\prime$. It follows that
$$P_{\l,\e}(t)\les P_{\l^\prime,\e}(t),
\q \Pi_{\l,\e}(t)\les \Pi_{\l^\prime,\e}(t),\qq \forall t\in[0,T],$$
and hence
$$0<P_\l(t)\les P_{\l^\prime}(t),\q 0<\Pi_\l(t)\les \Pi_{\l^\prime}(t),
\q \forall t\in[0,T];\qq 0<\l\les\l^\prime.$$
Therefore, the families $\{\Si_\l(t)\deq P_\l(t)^{-1}:\l>0\}$
and $\{\G_\l(t)\deq\Pi_\l(t)^{-1}:\l>0\}$ are decreasing in $\dbS^n_+$
and hence converge. We denote
$$\Si(t)=\lim_{\l\to\i}\Si_\l(t)\ges0,\q\G(t)=\lim_{\l\to\i}\G_\l(t)\ges0,\qq t\in[0,T].$$
Now using the fact
$$\left\{\2n\ba{ll}
\no\ss\ds{d\over dt}\big[P_\l(t)^{-1}P_\l(t)\big]=0,
\qq {d\over dt}\big[\Pi_\l(t)^{-1}\Pi_\l(t)\big]=0,\\
\no\ss\ds[R_1(t)+P_\l(t)]^{-1}=[I+P_\l(t)^{-1}R_1(t)]^{-1}P_\l(t)^{-1},\\
\ds[R_1(t)+\bar R_1(t)+P_\l(t)]^{-1}
=\big\{I+P_\l(t)^{-1}[R_1(t)+\bar R_1(t)]\big\}^{-1}P_\l(t)^{-1},
\ea\right.$$
one can easily show that $\Si_\l(\cd)$ is a solution of
\bel{16Mar22-17:24}\left\{\2n\ba{ll}
\ds\dot\Si_\l-A\Si_\l-\Si_\l A^\top-\Si_\l Q\Si_\l+BR_2^{-1}B^\top
+C(I+\Si_\l R_1)^{-1}\Si_\l C^\top=0,\\
\ns\ds \Si_\l(T)=\l^{-1} I,\ea\right.\ee
and $\G_\l(\cd)$ is a solution of
\bel{16Mar22-17:40}\lt\{\2n\ba{ll}
\ds \dot\G_\l-(A+\bar A)\G_\l-\G_\l(A+\bar A)^\top-\G_\l(Q+\bar Q)\G_\l
+(B+\bar B)(R_2+\bar R_2)^{-1}(B+\bar B)^\top\\
\ns\ds\q\1n~+(C+\bar C)\big[I+\Si_\l(R_1+\bar R_1)\big]^{-1}\Si_\l(C+\bar C)^\top=0,\\
\ns\ds \G_\l(T)=\l^{-1} I.\ea\rt.\ee
Note that \rf{16Mar22-17:24} is equivalent to
$$\Si_\l(t)=\l^{-1}I-\int_t^T\[A\Si_\l+\Si_\l A^\top\1n+\Si_\l Q\Si_\l
-BR_2^{-1}B^\top\1n-C(I+\Si_\l R_1)^{-1}\Si_\l C^\top\]ds.$$
Because $\{\Si_\l(t)\}_{\l\ges1}$ and $\{[I+\Si_\l(t)R_1(t)]^{-1}\Si_\l(t)\}_{\l\ges1}$
are uniformly bounded on $[0,T]$, by letting $\l\to\i$, we obtain from the dominated
convergence theorem:
$$\Si(t)=-\int_t^T\[A\Si+\Si A^\top+\Si Q\Si-BR_2^{-1}B^\top-C(I+\Si R_1)^{-1}\Si C^\top\]ds,$$
so $\Si(\cd)$ is a solution of the Riccati equation \rf{16Aug22-Ric-Si}.
Likewise, $\G(\cd)$ is a solution of the Riccati equation \rf{16Aug22-Ric-G}.

\ms

To prove the uniqueness, let us suppose that $\Si_1(\cd),\Si_2(\cd)\in C([0,T];\cl{\dbS^n_+})$
are two solutions of \rf{16Aug22-Ric-Si}.
Then it is easy to show that $\D(\cd)\deq\Si_1(\cd)-\Si_2(\cd)$ is a solution to the equation
$$\left\{\2n\ba{ll}
\ds\dot\D-(A+\Si_1Q)\D-\D(A+\Si_2Q)^\top+C(I+\Si_1 R_1)^{-1}\D\big[I-R_1(I+\Si_2R_1)^{-1}\Si_2\big]C^\top=0,\\
\ns\ds\D(T)=0.\ea\right.$$
Note that the functions $\Si_i$ and $(I+\Si_i R_1)^{-1},~i=1,2$ are bounded on $[0,T]$.
Then a standard argument using the Gronwall inequality will show that $\D(\cd)=0$.
The uniqueness of the solution to equation \rf{16Aug22-Ric-G} is proved similarly.
\end{proof}

\section{Representations of optimal controls and value function}

This section is going to give explicit formulas of the optimal controls and the value function,
via the solutions to the Riccati equations \rf{16Aug22-Ric-Si}, \rf{16Aug22-Ric-G},
and the MF-BSDE \rf{16Aug22-f}. Our first result can be stated as follows.

\begin{theorem}\label{bt-16Mar22-20:00}\sl
Let {\rm(H1)--(H2)} hold and let $\xi\in L^2_{\cF_T}(\O;\dbR^n)$ be given.
Let $\Si(\cd)$ and $\G(\cd)$ be the unique solutions to the Riccati equations \rf{16Aug22-Ric-Si}
and \rf{16Aug22-Ric-G}, respectively,
and let $(\f(\cd),\b(\cd))$ be the unique adapted solution to the MF-BSDE \rf{16Aug22-f}.
Then the following MF-FSDE admits a unique solution $X(\cd)$:
\bel{16Mar19-X}\left\{\2n\ba{ll}
\ds dX=\left\{-\,(A+\Si Q)^\top X-\big[\bar A-\Si Q+\G(Q+\bar Q)\big]^\top \dbE[X]-Q\f-\bar Q\dbE[\f]\right\}ds\\
\ns\ds\qq~~+\Big\{\big[R_1(I+\Si R_1)^{-1}\Si-I\big]C^\top X+\(-\bar C^\top-R_1(I+\Si R_1)^{-1}\Si C^\top\\
\ns\ds\qq\qq~+(R_1+\bar R_1)\big[I+\Si(R_1+\bar R_1)\big]^{-1}\Si(C+\bar C)^\top\)\dbE[X]\\
\ns\ds\qq\qq~
-R_1(I+\Si R_1)^{-1}(\b-\dbE[\b])-(R_1+\bar R_1)\big[I+\Si(R_1+\bar R_1)\big]^{-1}\dbE[\b]\Big\}dW,\\
\ns\ds X(t)=-[I+G\Si(t)]^{-1}G\{\f(t)-\dbE[\f(t)]\}-[I+(G+\bar G)\G(t)]^{-1}(G+\bar G)\dbE[\f(t)],
\ea\right.\ee
and the unique optimal control of Problem {\rm(MF-BSLQ)} for the terminal state $\xi$ is given by
\bel{16Mar19-u*}u=R_2^{-1}B^\top(X-\dbE[X])+(R_2+\bar R_2)^{-1}(B+\bar B)^\top\dbE[X].\ee
\end{theorem}

\begin{proof}
It is clear that \rf{16Mar19-X} has a unique solution $X(\cd)$.
So we need only prove that $u(\cd)$ defined by \rf{16Mar19-u*} is the unique optimal control
of Problem (MF-BSLQ) for the terminal state $\xi$. To this end, we define
\begin{eqnarray}
&&\label{16Aug25-Y} Y=-\,\Si(X-\dbE[X])-\G\dbE[X]-\f,\\
&&\label{16Aug25-Z} Z=(I+\Si R_1)^{-1}\big\{\Si C^\top(X-\dbE[X])-(\b-\dbE[\b])\big\}\\
&&\hphantom{Z=}+(I+\Si R_1+\Si\bar R_1)^{-1}\big\{\Si(C+\bar C)^\top\dbE[X]-\dbE[\b]\big\}.\nonumber
\end{eqnarray}
Then we have $Y(T)=\xi$ and
\begin{eqnarray}
&&\label{16Aug25-EY} \dbE[Y]=-\,\G\dbE[X]-\dbE[\f],\\
&&\label{16Aug25-EZ} \dbE[Z]=(I+\Si R_1+\Si\bar R_1)^{-1}\big\{\Si(C+\bar C)^\top\dbE[X]-\dbE[\b]\big\}.
\end{eqnarray}
Also, from \rf{16Aug25-Z} and \rf{16Aug25-EZ} we have
\bel{Z=}\ba{ll}
\ds Z+\Si R_1Z+\Si\bar R_1\dbE[Z]-\Si C^\top X-\Si\bar C^\top\dbE[X]+\b\\
\ns =(I+\Si R_1)Z+\Si\bar R_1\dbE[Z]-\Si C^\top X-\Si\bar C^\top\dbE[X]+\b\\
\ns =\Si C^\top(X-\dbE[X])-(\b-\dbE[\b])+(I+\Si R_1)\dbE[Z]\\
\ns\hphantom{=}+\Si\bar R_1\dbE[Z]-\Si C^\top X-\Si\bar C^\top\dbE[X]+\b\\
\ns =-\,\Si(C+\bar C)^\top\dbE[X]+\dbE[\b]+(I+\Si R_1+\Si\bar R_1)\dbE[Z]\\
\ns =0. \ea\ee
Thus, making use of \rf{16Aug22-Ric-Si}, \rf{16Aug22-Ric-G}, and \rf{Z=}, we have
\begin{eqnarray*}
dY &\4n=\4n&
-\,\dot\Si(X-\dbE[X])ds-\Si d(X-\dbE[X])-\dot\G\dbE[X]ds-\G d\dbE[X]-d\f\\
&\4n=\4n&
-\,\dot\Si(X-\dbE[X])ds+\Si\Big\{(A+\Si Q)^\top(X-\dbE[X])+Q(\f-\dbE[\f])\Big\}ds\\
&\4n~\4n&
-\,\Si\Big\{\big[R_1(I+\Si R_1)^{-1}\Si-I\big]C^\top X+\(\1n-\bar C^\top-R_1(I+\Si R_1)^{-1}\Si C^\top\\
&\4n~\4n&\hphantom{-\,\Si\Big\{}
+(R_1+\bar R_1)(I+\Si R_1+\Si\bar R_1)^{-1}\Si(C+\bar C)^\top\)\dbE[X]\\
&\4n~\4n&\hphantom{-\,\Si\Big\{}
-R_1(I+\Si R_1)^{-1}(\b-\dbE[\b])-(R_1+\bar R_1)(I+\Si R_1+\Si\bar R_1)^{-1}\dbE[\b]\Big\}dW\\
&\4n~\4n&
-\,\dot\G\dbE[X]ds+\G\Big\{\big[A+\bar A+\G(Q+\bar Q)\big]^\top \dbE[X]+(Q+\bar Q)\dbE[\f]\Big\}ds\\
&\4n~\4n&
-\,\Big\{(A+\Si Q)\f+\big[\bar A+\G(Q+\bar Q)-\Si Q\big]\dbE[\f]+C(I+\Si R_1)^{-1}\b\\
&\4n~\4n&\hphantom{-\,\Big\{}
+\big[(C+\bar C)(I+\Si R_1+\Si\bar R_1)^{-1}-C(I+\Si R_1)^{-1}\big]\dbE[\b]\Big\}ds-\b dW\\
&\4n=\4n&
\Big\{\(\1n-\dot\Si+\Si(A+\Si Q)^\top\)(X-\dbE[X])+\(-\dot\G+\G\big[A+\bar A+\G(Q+\bar Q)\big]^\top\)\dbE[X]\\
&\4n~\4n&\hphantom{\Big\{}
-A\f-\bar A\dbE[\f]-C(I+\Si R_1)^{-1}(\b-\dbE[\b])-(C+\bar C)(I+\Si R_1+\Si\bar R_1)^{-1}\dbE[\b]\Big\}ds\\
&\4n~\4n&
-\,\Big\{\Si R_1(I+\Si R_1)^{-1}\Si C^\top(X-\dbE[X])-\Si C^\top X-\Si\bar C^\top\dbE[X]\\
&\4n~\4n&\hphantom{-\,\Big\{}
+\Si(R_1+\bar R_1)(I+\Si R_1+\Si\bar R_1)^{-1}\Si(C+\bar C)^\top\dbE[X]-\Si R_1(I+\Si R_1)^{-1}(\b-\dbE[\b])\\
&\4n~\4n&\hphantom{-\,\Big\{}
-\Si(R_1+\bar R_1)(I+\Si R_1+\Si\bar R_1)^{-1}\dbE[\b]+\b\Big\}dW\\
&\4n=\4n&
\Big\{\(\1n-A\Si+BR_2^{-1}B^\top+C(I+\Si R_1)^{-1}\Si C^\top\)(X-\dbE[X])+\(-(A+\bar A)\G\\
&\4n~\4n&\hphantom{\Big\{}
+(B+\bar B)(R_2+\bar R_2)^{-1}(B+\bar B)^\top+(C+\bar C)(I+\Si R_1+\Si\bar R_1)^{-1}\Si(C+\bar C)^\top\)\dbE[X]\\
&\4n~\4n&\hphantom{\Big\{}
-A\f-\bar A\dbE[\f]-C(I+\Si R_1)^{-1}(\b-\dbE[\b])-(C+\bar C)(I+\Si R_1+\Si\bar R_1)^{-1}\dbE[\b]\Big\}ds\\
&\4n~\4n&
-\,\Big\{\Si R_1(I+\Si R_1)^{-1}\big\{\Si C^\top(X-\dbE[X])-(\b-\dbE[\b])\big\}
-\Si C^\top X-\Si\bar C^\top\dbE[X]+\b\\
&\4n~\4n&\hphantom{-\,\Big\{}
+\Si(R_1+\bar R_1)(I+\Si R_1+\Si\bar R_1)^{-1}\big\{\Si(C+\bar C)^\top\dbE[X]-\dbE[\b]\big\}\Big\}dW\\
&\4n=\4n&
\Big\{\1n-A\(\Si(X-\dbE[X])+\G\dbE[X]+\f\)-\bar A\(\G\dbE[X]+\dbE[\f]\)+BR_2^{-1}B^\top(X-\dbE[X])\\
&\4n~\4n&\hphantom{\Big\{}
+(B\1n+\1n\bar B)(R_2\1n+\1n\bar R_2)^{-1}(B\1n+\1n\bar B)^\top\dbE[X]\1n
+\1nC(I\1n+\1n\Si R_1)^{-1}\big\{\Si C^\top(X\1n-\1n\dbE[X])\1n-\1n(\b\1n-\1n\dbE[\b])\big\}\\
&\4n~\4n&\hphantom{\Big\{}
+(C+\bar C)(I+\Si R_1+\Si\bar R_1)^{-1}\big\{\Si(C+\bar C)^\top\dbE[X]-\dbE[\b]\big\}\Big\}ds\\
&\4n~\4n&
-\,\Big\{\Si R_1(Z-\dbE[Z])-\Si C^\top X-\Si\bar C^\top\dbE[X]+\b+\Si(R_1+\bar R_1)\dbE[Z]\Big\}dW\\
&\4n=\4n&
\Big\{AY+\bar A\dbE[Y]+B(u-\dbE[u])+(B+\bar B)\dbE[u]+C(Z-\dbE[Z])+(C+\bar C)\dbE[Z]\Big\}ds\\
&\4n~\4n&
-\,\Big\{\Si R_1Z+\Si\bar R_1\dbE[Z]-\Si C^\top X-\Si\bar C^\top\dbE[X]+\b\Big\}dW\\
&\4n=\4n&
\Big\{AY+\bar A\dbE[Y]+Bu+\bar B\dbE[u]+CZ+\bar C\dbE[Z]\Big\}ds+ZdW.
\end{eqnarray*}
Moreover, the first equation in \rf{16Mar19-X} can be written as
\begin{eqnarray*}
dX&\4n=\4n&
\Big\{\1n-A^\top X-\bar A^\top\dbE[X]-Q\(\Si(X-\dbE[X])+\G\dbE[X]+\f\)
-\bar Q\(\G\dbE[X]+\dbE[\f]\)\Big\}ds\\
&\4n~\4n&
+\,\Big\{\1n-C^\top X-\bar C^\top\dbE[X]+R_1(I+\Si R_1)^{-1}\(\Si C^\top(X-\dbE[X])-(\b-\dbE[\b])\)\\
&\4n~\4n&\hphantom{+\,\Big\{}
+(R_1+\bar R_1)\big[I+\Si(R_1+\bar R_1)\big]^{-1}\(\Si(C+\bar C)^\top\dbE[X]-\dbE[\b]\)\Big\}dW\\
&\4n=\4n&
\Big\{\1n-A^\top X-\bar A^\top\dbE[X]+QY+\bar Q\dbE[Y]\Big\}ds
+\Big\{\1n-C^\top X-\bar C^\top\dbE[X]+R_1Z+\bar R_1\dbE[Z]\Big\}dW.
\end{eqnarray*}
From the second equation in \rf{16Mar19-X}, we see
\begin{eqnarray}
\label{16Aug26:EX(t)}\dbE[X(t)]&\3n=\3n&-[I+(G+\bar G)\G(t)]^{-1}(G+\bar G)\dbE[\f(t)],\\
\label{16Aug26:X(t)-EX(t)}X(t)-\dbE[X(t)]&\3n=\3n&-[I+G\Si(t)]^{-1}G\{\f(t)-\dbE[\f(t)]\}.
\end{eqnarray}
\rf{16Aug25-EY} and \rf{16Aug26:EX(t)} yield
$$[I+(G+\bar G)\G(t)]\dbE[X(t)]=-(G+\bar G)\dbE[\f(t)]=(G+\bar G)\big\{\G\dbE[X(t)]+\dbE[Y(t)]\big\},$$
from which follows
\bel{16Aug25-16:20-1}\dbE[X(t)]=(G+\bar G)\dbE[Y(t)].\ee
Note that by \rf{16Aug25-Y} and \rf{16Aug25-EY},
$$Y(t)-\dbE[Y(t)]=-\Si(t)\big\{X(t)-\dbE[X(t)]\big\}-\big\{\f(t)-\dbE[\f(t)]\big\},$$
which, together with \rf{16Aug26:X(t)-EX(t)}, yields
\begin{eqnarray*}
[I+G\Si(t)]\big\{X(t)-\dbE[X(t)]\big\} &\3n=\3n& -\,G\big\{\f(t)-\dbE[\f(t)]\big\}\\
&\3n=\3n& G\(\Si(t)\big\{X(t)-\dbE[X(t)]\big\}+Y(t)-\dbE[Y(t)]\),
\end{eqnarray*}
from which follows
\bel{16Aug25-16:20-2}X(t)-\dbE[X(t)]=G\big\{Y(t)-\dbE[Y(t)]\big\}.\ee
Combining \rf{16Aug25-16:20-1}--\rf{16Aug25-16:20-2} we have
$$X(t)=GY(t)+\bar G\dbE[Y(t)].$$
Finally, observing that $u(\cd)$ defined by \rf{16Mar19-u*} satisfies
$$ R_2u+\bar R_2\dbE[u]-B^\top X-\bar B^\top \dbE[X]=0,$$
we see that $(X(\cd),Y(\cd),Z(\cd),u(\cd))$ solves the optimality system
\rf{16Aug22-16:20}--\rf{16Aug22-16:30}.
The result then follows immediately from Theorem \ref{bt-16Aug22-16:06}.
\end{proof}

We next present a formula for the value function of Problem (MF-BSLQ).

\begin{theorem}\label{bt-16Mar26-20:00}\sl
Let {\rm(H1)--(H2)} hold. Then the value function of Problem {\rm(MF-BSLQ)} is given by
\begin{eqnarray*}
V(t,\xi)\4n&=&\4n \dbE\Big\{\blan G[I+\Si(t)G]^{-1}(\f(t)-\dbE[\f(t)]),\f(t)-\dbE[\f(t)]\bran\\
\4n&~&\4n\hphantom{\dbE\Big\{}
+\blan(G+\bar G)[I+\G(t)(G+\bar G)]^{-1}\dbE[\f(t)],\dbE[\f(t)]\bran\\
\4n&~&\4n\hphantom{\dbE\Big\{}
+\int_t^T\[\lan Q(\f-\dbE[\f]),\f-\dbE[\f]\ran
+\blan(Q+\bar Q)\dbE[\f],\dbE[\f]\bran\\
\4n&~&\4n\hphantom{\dbE\Big\{+\int_t^T\[}
+\blan(I+R_1\Si)^{-1}R_1(\b-\dbE[\b]),\b-\dbE[\b]\bran\\
\4n&~&\4n\hphantom{\dbE\Big\{+\int_t^T\[}
+\blan[I+(R_1+\bar R_1)\Si]^{-1}(R_1+\bar R_1)\dbE[\b],\dbE[\b]\bran\]ds\Big\}.
\end{eqnarray*}
where $\Si(\cd)$ and $\G(\cd)$ are the unique solutions to the Riccati equations
\rf{16Aug22-Ric-Si} and \rf{16Aug22-Ric-G}, respectively, and $(\f(\cd),\b(\cd))$
is the unique adapted solution to the MF-BSDE \rf{16Aug22-f}.
\end{theorem}

\begin{proof}
Let $(Y^*(\cd),Z^*(\cd),u^*(\cd))$ be the optimal triple corresponding to the terminal state
$\xi\in L^2_{\cF_T}(\O;\dbR^n)$, and let $X^*(\cd)$ be the solution to MF-FSDE \rf{16Aug22-16:00}.
According to Theorem \ref{bt-16Aug22-16:06}, $(X^*(\cd),Y^*(\cd),Z^*(\cd),u^*(\cd))$ satisfies
the optimality system \rf{16Aug22-16:20}--\rf{16Aug22-16:30}.
On the other hand, let $X(\cd)$ be the solution to \rf{16Mar19-X}, and let $u(\cd)$, $Y(\cd)$,
and $Z(\cd)$ be defined by \rf{16Mar19-u*}, \rf{16Aug25-Y}, and \rf{16Aug25-Z}, respectively.
We recall from the proof of Theorem \ref{bt-16Mar22-20:00} that $(X(\cd),Y(\cd),Z(\cd),u(\cd))$
also satisfies the optimality system \rf{16Aug22-16:20}--\rf{16Aug22-16:30}.
By the uniqueness of optimal controls, we must have
$$(X^*(\cd),Y^*(\cd),Z^*(\cd),u^*(\cd))=(X(\cd),Y(\cd),Z(\cd),u(\cd)).$$
Thus, the value $V(t,\xi)$ is equal to
\begin{eqnarray*}
J(t,\xi;u(\cd))\4n&=&\4n \dbE\bigg\{\int_t^T\[\lan QY,Y\ran+\blan\bar Q\dbE[Y],\dbE[Y]\bran
+\lan R_1Z,Z\ran+\blan\bar R_1\dbE[Z],\dbE[Z]\bran\\
\4n&~&\4n\hphantom{\dbE\bigg\{}
+\lan R_2u,u\ran+\blan\bar R_2\dbE[u],\dbE[u]\bran\]ds
+\lan GY(t),Y(t)\ran+\blan\bar G\dbE[Y(t)],\dbE[Y(t)]\bran\bigg\}\\
\4n&=&\4n\dbE\bigg\{\int_t^T\[\lan Q(Y-\dbE[Y]),Y-\dbE[Y]\ran+\blan(Q+\bar Q)\dbE[Y],\dbE[Y]\bran\\
\4n&~&\4n\hphantom{\dbE\bigg\{}
+\lan R_1(Z-\dbE[Z]),Z-\dbE[Z]\ran+\blan(R_1+\bar R_1)\dbE[Z],\dbE[Z]\bran\\
\4n&~&\4n\hphantom{\dbE\bigg\{}
+\lan R_2(u-\dbE[u]),u-\dbE[u]\ran+\blan(R_2+\bar R_2)\dbE[u],\dbE[u]\bran\]ds\\
\4n&~&\4n\hphantom{\dbE\bigg\{}
+\lan G(Y(t)-\dbE[Y(t)]),Y(t)-\dbE[Y(t)]\ran
+\blan(G+\bar G)\dbE[Y(t)],\dbE[Y(t)]\bran\bigg\}.
\end{eqnarray*}
Noting that
\begin{eqnarray*}
\dbE[Y] &\3n=\3n& -\G\dbE[X]-\dbE[\f],\\
\dbE[Z] &\3n=\3n& \big[I+\Si (R_1+\bar R_1)\big]^{-1}\big\{\Si(C+\bar C)^\top\dbE[X]-\dbE[\b]\big\},\\
\dbE[u] &\3n=\3n& (R_2+\bar R_2)^{-1}(B+\bar B)^\top\dbE[X],\\
Y-\dbE[Y] &\3n=\3n& -\Si(X-\dbE[X])-(\f-\dbE[\f]),\\
Z-\dbE[Z] &\3n=\3n& (I+\Si R_1)^{-1}\big\{\Si C^\top(X-\dbE[X])-(\b-\dbE[\b])\big\},\\
u-\dbE[u] &\3n=\3n& R_2^{-1}B^\top(X-\dbE[X]),
\end{eqnarray*}
and using the fact that
$$(I+MN)^{-1}M = M(I+NM)^{-1},\qq \forall M,N\in\cl{\dbS^n_+},$$
it can be shown by a straightforward computation that
{\small\bel{16Mar27-18:00}\ba{ll}
\ds\dbE\int_t^T\[\lan Q(Y\1n-\1n\dbE[Y]),Y\1n-\1n\dbE[Y]\ran
\1n+\1n\lan R_1(Z\1n-\1n\dbE[Z]),Z\1n-\1n\dbE[Z]\ran
\1n+\1n\lan R_2(u\1n-\1n\dbE[u]),u\1n-\1n\dbE[u]\ran\]ds\\
\ns\ds=\dbE\int_t^T\bigg\{\Blan\[\Si Q\Si\1n+\1nC(I\1n+\1n\Si R_1)^{-1}\Si R_1\Si
(I\1n+\1nR_1\Si)^{-1}C^\top\2n+\1nBR_2^{-1}B^\top\](X\1n-\1n\dbE[X]),X\1n-\1n\dbE[X]\Bran\\
\ns\ds\hphantom{=}+2\lan X\1n-\1n\dbE[X],\Si Q(\f\1n-\1n\dbE[\f])\ran
-2\blan X-\dbE[X],C(I+\Si R_1)^{-1}\Si R_1(I+\Si R_1)^{-1}(\b-\dbE[\b])\bran\\
\ns\ds\hphantom{=}+\lan Q(\f-\dbE[\f]),\f-\dbE[\f]\ran
+\blan(I+R_1\Si)^{-1}R_1(I+\Si R_1)^{-1}(\b-\dbE[\b]),\b-\dbE[\b]\bran\bigg\}ds, \ea\ee}
and that
{\small\bel{16Mar27-18:10}\ba{ll}
\ds\int_t^T\[\blan(Q+\bar Q)\dbE[Y],\dbE[Y]\bran+\blan(R_1+\bar R_1)\dbE[Z],\dbE[Z]\bran
+\blan(R_2+\bar R_2)\dbE[u],\dbE[u]\bran\]ds\\
\ns\ds=\int_t^T\bigg\{\Blan\[\G(Q\1n+\1n\bar Q)\G
\1n+\1n(C\1n+\1n\bar C)\big[I\1n+\1n\Si(R_1\1n+\1n\bar R_1)\big]^{-1}
\Si(R_1\1n+\1n\bar R_1)\Si\big[I\1n+\1n(R_1\1n+\1n\bar R_1)\Si\big]^{-1}(C\1n+\1n\bar C)^\top\\
\ns\ds\qq\qq~~+(B+\bar B)(R_2+\bar R_2)^{-1}(B+\bar B)^\top\]\dbE[X],\dbE[X]\Bran
+2\blan\dbE[X],\G(Q+\bar Q)\dbE[\f]\bran\\
\ns\ds\hphantom{=}
-2\blan\dbE[X],(C+\bar C)\big[I+\Si(R_1+\bar R_1)\big]^{-1}
\Si(R_1+\bar R_1)\big[I+\Si(R_1+\bar R_1)\big]^{-1}\dbE[\b]\bran\\
\ns\ds\hphantom{=}
+\blan(Q\1n+\1n\bar Q)\dbE[\f],\dbE[\f]\bran+\blan\big[I\1n+\1n(R_1\1n+\1n\bar R_1)\Si\big]^{-1}
(R_1\1n+\1n\bar R_1)\big[I\1n+\1n\Si(R_1\1n+\1n\bar R_1)\big]^{-1}\dbE[\b],\dbE[\b]\bran\bigg\}ds. \ea\ee}
Observing that
$$\left\{\2n\ba{ll}
\ds d\dbE[X]=-\Big\{\big[A+\bar A+\G(Q+\bar Q)\big]^\top\dbE[X]+(Q+\bar Q)\dbE[\f]\Big\}ds,\\
\ns d(X-\dbE[X])=-\Big\{(A+\Si Q)^\top(X-\dbE[X])+Q(\f-\dbE[\f])\Big\}ds\\
\ns\q~-\Big\{(I+R_1\Si)^{-1}C^\top(X-\dbE[X])+\big[I+(R_1+\bar R_1)\Si\big]^{-1}(C+\bar C)^\top\dbE[X]\\
\ns\qq\q~+(I+R_1\Si)^{-1}R_1(\b-\dbE[\b])+\big[I+(R_1+\bar R_1)\Si\big]^{-1}(R_1+\bar R_1)\dbE[\b]\Big\}dW,
\ea\right.$$
we have by applying It\^{o}'s formula to $s\mapsto\lan\Si(s)(X(s)-\dbE[X(s)]),X(s)-\dbE[X(s)]\ran$,
\bel{16Mar27-18:20}\ba{ll}
-\,\dbE\blan\Si(t)\big\{X(t)-\dbE[X(t)]\big\},X(t)-\dbE[X(t)]\bran\\
\ns\ds=\dbE\int_t^T\Big\{
\Blan\[\dot\Si-(A+\Si Q)\Si-\Si(A+\Si Q)^\top+C(I+\Si R_1)^{-1}\Si(I+R_1\Si)^{-1}C^\top\]\\
\ns\ds\hphantom{=\dbE\int_t^T\Big\{\Blan}
\cd(X-\dbE[X]),X-\dbE[X]\Bran-2\lan X-\dbE[X],\Si Q(\f-\dbE[\f])\ran\\
\ns\ds\hphantom{=\dbE\int_t^T\Big\{\Blan}
+2\blan X-\dbE[X],C(I+\Si R_1)^{-1}\Si R_1(I+\Si R_1)^{-1}(\b-\dbE[\b])\bran\\
\ns\ds\hphantom{=\dbE\int_t^T\Big\{\Blan}
+\lan(I+R_1\Si)^{-1}R_1\Si R_1(I+\Si R_1)^{-1}(\b-\dbE[\b]),\b-\dbE[\b]\ran\Big\}ds
\ea\ee 
$$\ba{ll}
\ds\hphantom{=}
+\int_t^T\Big\{\blan(C+\bar C)\big[I+\Si(R_1+\bar R_1)\big]^{-1}
\Si\big[I+(R_1+\bar R_1)\Si\big]^{-1}(C+\bar C)^\top\dbE[X],\dbE[X]\bran\\
\ns\ds\hphantom{=}
+2\blan\dbE[X],(C+\bar C)\big[I+\Si(R_1+\bar R_1)\big]^{-1}
\Si(R_1+\bar R_1)\big[I+\Si(R_1+\bar R_1)\big]^{-1}\dbE[\b]\bran\\
\ns\ds\hphantom{=}
+\blan\big[I+(R_1+\bar R_1)\Si\big]^{-1}
(R_1+\bar R_1)\Si(R_1+\bar R_1)\big[I+\Si(R_1+\bar R_1)\big]^{-1}\dbE[\b],\dbE[\b]\bran\Big\}ds,
\ea$$ 
and by applying the integration by parts formula to $s\mapsto\lan\G(s)\dbE[X(s)],\dbE[X(s)]\ran$,
we have
\bel{16Mar27-18:30}\ba{ll}
-\,\lan\G(t)\dbE[X(t)],\dbE[X(t)]\ran\\
\ns\ds=\int_t^T\Big\{\Blan\(\dot\G-\big[A+\bar A+\G(Q+\bar Q)\big]\G
-\G\big[A+\bar A+\G(Q+\bar Q)\big]^\top\)\dbE[X],\dbE[X]\Bran\\
\ns\ds\hphantom{=\int_t^T\Big\{}
-2\blan\dbE[X],\G(Q+\bar Q)\dbE[\f]\bran\Big\}ds.\ea\ee
Now adding equations \rf{16Mar27-18:00}, \rf{16Mar27-18:10}, \rf{16Mar27-18:20}
and \rf{16Mar27-18:30} yields
\bel{16Aug28-15:44}\ba{lll}
\ds V(t,\xi)\4n&=&\4n \dbE\Big\{\lan G(Y(t)-\dbE[Y(t)]),Y(t)-\dbE[Y(t)]\ran
+\blan(G+\bar G)\dbE[Y(t)],\dbE[Y(t)]\bran\\
\ns\4n&~&\4n\ds\hphantom{\dbE\Big\{}
+\blan\Si(t)\big\{X(t)-\dbE[X(t)]\big\},X(t)-\dbE[X(t)]\bran
+\lan\G(t)\dbE[X(t)],\dbE[X(t)]\ran\Big\}\\
\ns\4n&~&\4n\ds+\,\dbE\int_t^T\Big\{\lan Q(\f-\dbE[\f]),\f-\dbE[\f]\ran
+\blan(Q+\bar Q)\dbE[\f],\dbE[\f]\bran\\
\ns\4n&~&\4n\ds\hphantom{+\,\dbE\int_t^T\Big\{}
+\blan(I+R_1\Si)^{-1}R_1(\b-\dbE[\b]),\b-\dbE[\b]\bran\\
\ns\4n&~&\4n\ds\hphantom{+\,\dbE\int_t^T\Big\{}
+\blan\big[I+(R_1+\bar R_1)\Si\big]^{-1}(R_1+\bar R_1)\dbE[\b],\dbE[\b]\bran\Big\}ds.
\ea\ee
Recalling that
$$\dbE[Y]=-\G\dbE[X]-\dbE[\f],\qq Y-\dbE[Y]=-\Si(X-\dbE[X])-(\f-\dbE[\f]),$$
and noting that
\begin{eqnarray*}
\dbE[X(t)]\3n&=&\3n-[I+(G+\bar G)\G(t)]^{-1}(G+\bar G)\dbE[\f(t)],\\
X(t)-\dbE[X(t)]\3n&=&\3n-[I+G\Si(t)]^{-1}G\{\f(t)-\dbE[\f(t)]\},
\end{eqnarray*}
we obtain
\begin{eqnarray*}
&&\dbE\Big\{\lan G(Y(t)-\dbE[Y(t)]),Y(t)-\dbE[Y(t)]\ran
+\blan(G+\bar G)\dbE[Y(t)],\dbE[Y(t)]\bran\\
&&\hphantom{\dbE\Big\{}
+\blan\Si(t)\big\{X(t)-\dbE[X(t)]\big\},X(t)-\dbE[X(t)]\bran
+\lan\G(t)\dbE[X(t)],\dbE[X(t)]\ran\Big\}\\
&&=\dbE\Big\{\blan G\big\{\Si(t)(X(t)\1n-\1n\dbE[X(t)])\1n+\1n(\f(t)\1n-\1n\dbE[\f(t)])\big\},
\Si(t)(X(t)\1n-\1n\dbE[X(t)])\1n+\1n(\f(t)\1n-\1n\dbE[\f(t)])\bran\\
&&\hphantom{=\dbE\Big\{}
+\blan(G+\bar G)\big\{\G(t)\dbE[X(t)]+\dbE[\f(t)]\big\},\G(t)\dbE[X(t)]+\dbE[\f(t)]\bran\\
&&\hphantom{=\dbE\Big\{}
+\blan\Si(t)\big\{X(t)-\dbE[X(t)]\big\},X(t)-\dbE[X(t)]\bran
+\lan\G(t)\dbE[X(t)],\dbE[X(t)]\ran\Big\}\\
&&=\dbE\Big\{\lan\Si(t)[I+G\Si(t)](X(t)-\dbE[X(t)]),X(t)-\dbE[X(t)]\ran\\
&&\hphantom{=\dbE\Big\{}
+2\lan G\Si(t)(X(t)-\dbE[X(t)]),\f(t)-\dbE[\f(t)]\ran
+\lan G(\f(t)-\dbE[\f(t)]),\f(t)-\dbE[\f(t)]\ran\\
&&\hphantom{=\dbE\Big\{}
+\blan\G(t)[I+(G+\bar G)\G(t)]\dbE[X(t)],\dbE[X(t)]\bran
+2\blan(G+\bar G)\G(t)\dbE[X(t)],\dbE[\f(t)]\bran\\
&&\hphantom{=\dbE\Big\{}
+\blan(G+\bar G)\dbE[\f(t)],\dbE[\f(t)]\bran\Big\}\\
&&=\dbE\Big\{\blan G[I+\Si(t)G]^{-1}(\f(t)-\dbE[\f(t)]),\f(t)-\dbE[\f(t)]\bran\\
&&\hphantom{=\dbE\Big\{}
+\blan(G+\bar G)[I+\G(t)(G+\bar G)]^{-1}\dbE[\f(t)],\dbE[\f(t)]\bran\Big\}.
\end{eqnarray*}
Substitution of the above into \rf{16Aug28-15:44} completes the proof.
\end{proof}


\end{document}